\documentclass{birkjour}
 \usepackage[utf8]{inputenc}
 \usepackage{amsmath}
 \usepackage{amssymb}
 \usepackage{amsthm}
 \usepackage{verbatim}
\usepackage{multicol}
 \theoremstyle{definition}
 \newtheorem{definition}{Definition}[section]
 \newtheorem{remark}[definition]{Remark}
 \newtheorem{theorem}[definition]{Theorem}
 \newtheorem{example}[definition]{Example}
 \newtheorem{corollary}[definition]{Corollary}
 \newtheorem{lemma}[definition]{Lemma}

\usepackage{bbold}
 \usepackage{tikz}
 \usetikzlibrary{shapes,arrows}
 \usepackage[english]{babel}
 \usepackage{graphicx}
 \usepackage{braket}
 \usepackage{float}

 \title{\textbf{Strong solutions of semilinear SPDEs with unbounded diffusion}}
 \author{Florian Bechtold\\
 florian.bechtold@sorbonne-universite.fr}
 \date{\today}
 \address{LPSM, Sorbonne Université\\
 4 Place Jussieu\\
 75005 Paris\\
 France}
  \email{florian.bechtold@sorbonne-universite.fr}
\newcommand{\norm}[1]{\left\lVert#1\right\rVert}
\newcommand{\R}{\mathbb{R}}

\begin{document}

 \begin{abstract}
     We prove a modification to the classical maximal inequality for stochastic convolutions in 2-smooth Banach spaces using the factorization method. This permits to study semilinear stochastic partial differential equations with unbounded diffusion operators driven by cylindrical Brownian motion via the mild solution approach. In the case of finite dimensional driving noise, provided sufficient regularity on the coefficients, we establish existence and uniqueness of strong solutions. In the case of "critical unboundedness", we show how to use the stochastic compactness method to obtain a martingale solution.
 \end{abstract}
   \maketitle
 \section{Introduction}
 We consider the problem
 \begin{align}
 \label{main_eqn}
    \begin{array}{lll}
    du_t&=(\Delta-1) u_t dt+(-\Delta+1)^{\delta_0}F(u_t)dt+(-\Delta+1)^{\delta_1/2}B(u_t)dW_t \\
    u(0)&=u_0 
     \end{array}
\end{align}
 for $t\in [0,T]$ where $\delta_0, \delta_1\in [0,1)$, $\Delta$ is the Laplacian with periodic boundary conditions on the $N$-dimensional torus $\mathbb{T}^N$and $W$ is cylindrical Brownian motion over a separable Hilbert space $U$.
 \newline
 \indent
 A canonical setting for the investigation of well-posedness to the above problem would be assuming $F:L^2(\mathbb{T}^N)\rightarrow L^2(\mathbb{T}^N)$ and $B:L^2(\mathbb{T}^N)\rightarrow L_2(U, L^2(\mathbb{T}^N))$ to be Lipschitz. Yet, due to the unbounded fractional Laplacian in the diffusion term, one loses the Hilbert-Schmidt property and is thus unable to define a stochastic integral with values in $L^2(\mathbb{T}^N)$, as required in the variational approach to stochastic partial differential equations \cite{krylov_stochastic_1981}, \cite{Liu2015}, \cite{GYONGY1998271}. 
 \newline
 \indent
 In keeping the above assumptions, an alternative approach consists in the mild formulation to the problem given by
 \[
 u_t=S_tu_0+\int_0^tS_{t-s}(-\Delta+1)^{\delta_0}F(u_s)ds+\int_0^tS_{t-s}(-\Delta+1)^{\delta_1/2}B(u_s)dW_s,
 \]
 where $S$ denotes the semigroup generated by $\Delta-1$. Exploiting this formulation in the particular case of a bounded diffusion, i.e. $\delta_1=0$, Hofmanová is able to apply a fixed point theorem in $L^2(\Omega\times[0,T], L^2(\mathbb{T}^N)))$ \cite{Hofmanova2013}. To this end, she crucially exploits the fact that the fractional power of a generator $A$ composed with its corresponding semigroup $(S_t)_{t\geq 0}$ yields a bounded operator, or more precisely
 \[
 \norm{(-A)^\delta S_t}\leq C_\delta t^{-\delta},
 \]
 provided the semigroup is analytic \cite{Pazy1983}. Applying this bound in combination with the triangle inequality on the drift term, she is able to close the argument in a classical way. In particular, this encompasses the use of the maximal inequality for stochastic convolutions, which in this context reads
 \[
 \mathbb{E}\left[\sup_{t\leq T} \norm{\int_0^tS_{t-s} B(u_s)dW_s}^2_{L^2(\mathbb{T}^N)}\right]\leq C \mathbb{E}\left[\int_0^T \norm{B(u_s)}^2_{L_2(U, L^2(\mathbb{T}^N))}ds\right].
 \]
 Notice that in the present context of an unbounded diffusion coefficient, i.e. $\delta_1\neq 0$, the above bound is not available since the diffusion term is no longer Hilbert-Schmidt, therefore preventing an immediate generalization of \cite{Hofmanova2013}.
 \newline\indent
 In this article, we show how to generalise the above maximal inequality in a weaker form to the present context of an unbounded diffusion coefficient, notably
 \begin{align*}
      &\mathbb{E}\left[ \sup_{t\leq T} \norm{\int_0^t (-\Delta+1)^{\delta_1/2}S_{t-s} B(u_s)dW_s}_{L^p(\mathbb{T}^N)}^p\right]\\
      &\leq CT^{p/2(1-\delta_1)}\mathbb{E}\left[\sup_{t\leq T} \norm{B(u_t)}^p_{\gamma(U, L^p(\mathbb{T}^N))}\right]
 \end{align*}
 for $p>2/(1-\delta_1)$. The reasoning we follow exploits a combination of the above trade-off between fractional powers of generators and their analytic semigroups as well as the factorization method \cite{factorization}, \cite{daprato_zabczyk}. In particular, the new maximal inequality derived permits to set up a fixed point argument in $L^p(\Omega, C([0,T], L^p(\mathbb{T}^N)))$ , which establishes existence and uniqueness of mild solutions to the equation in question.\newline\indent
 In the particular case of finite dimensional noise, i.e. $U$ being of finite dimension, we show that the corresponding sequence of Picard iterations is uniformly bounded in Sobolev spaces, provided sufficient regularity of $F$ and $B$. By using the Sobolev embedding theorem, this allows to identify the constructed mild solutions as strong ones. In order to use the embedding theorem, we have be able to pass to Sobolev spaces of high enough order. This necessitates to leave the Hilbert space framework and consider the more general theory of stochastic integration in $2$-smooth Banach spaces, which is why our maximal inequality is stated in this more general setting. For a concise introduction to this theory, we refer the reader to \cite{ondrejat} and \cite{Brzeniak1995}. Let us mention that in the case of finite dimensional noise an alternative approach using rough path theory to define and control the stochastic convolution and subsequently solve the associated semilinear problem was recently presented in \cite{gerasimovics2019nonautonomous} building upon earlier results in \cite{gubinelli2010}. Note that the constraint $\delta_0, \delta_1\in [0,1)$ ensures that the problem \eqref{main_eqn} is subcritical in the terminology of \cite{gerasimovics2019nonautonomous}. \newline\indent
 Finally, we discuss the limit case $\delta_1\nearrow 1$, to which the present maximal inequality can not be applied. We show that provided $B$ is sufficiently small, the corresponding sequence of strong solutions can be used to establish existence of a martingale solution of the corresponding limiting equation via the stochastic compactness method due to \cite{flandoli1995martingale}.

\section{Setting and main results}
Throughout this article we fix a stochastic basis $(\Omega, \mathcal{F}, (\mathcal{F}_t)_{t\geq0}, \mathbb{P})$ with the filtration satisfying the usual conditions. Let $U$ be a separable Hilbert space and $W$ a cylindrical Brownian motion on $U$. In case $U$ is of finite dimension, this means that for any orthonormal basis $(e_i)_{i\leq d}$ of $U$, $W$ admits the expansion
\[
W_t=\sum_{i=1}^d e_i \beta^i_t
\]
where $(\beta^i)_{i\leq d}$ are independent standard Brownian motions. Let $X$ be a 2-smooth Banach space and denote by $\gamma(U, X)$ the space of $\gamma$-radonifying operators from $U$ to $X$. Let $\mathbb{T}^N$ denote the $N$-dimensional torus and $L^p(\mathbb{T}^N)$, $W^{m,p}(\mathbb{T}^N)$, $H^{\alpha, p}(\mathbb{T^N})$ the associated Lebesgue, respectively Sobolev, respectively Bessel potential spaces, which we recall all fall into the class of 2-smooth Banach spaces  for $p\geq 2$ \cite{ondrejat}. We fix moreover a finite time horizon $T\in \R^+$.\newline
\indent
For convenience we introduce for $q\geq 2$ and $X$ a separable Banach space the space
\[
Z_{q,X}:=L^q(\Omega, C([0,T], X)),
\]
which endowed with its naturally inherited norm
\[
\norm{u}_{Z_{q,X}}^q:=\mathbb{E}\left[\sup_{t\leq T} \norm{u_t}_X^q\right]
\]
is itself a separable Banach space. Throughout the article $C$ shall denote an unessential constant that may change from one line to the next. Dependencies of the constant $C$ on parameters are indicated by corresponding subscripts.

\begin{theorem}[A maximal inequality]
 \label{max_ineq}
    Let $\delta\in [0,1)$ and $T>0$. Let $X$ be a 2-smooth Banach space, $A: D(A)\subset X\rightarrow X$ be generator of an analytic contraction semigroup of operators $(S_t)_{t\geq 0}$.  Let $W$ be a cylindrical Wiener process on a separable Hilbert space $U$. Suppose a measurable $B: X\rightarrow \gamma(U,X)$ satisfies for $q>\frac{2}{1-\delta}$
    \[
  \norm{B(u)}_{\gamma(U,X)}^q\leq C(1+\norm{u}_{X}^q),
    \]
then for every progressively measurable $u\in Z_{q, X}$   the process
    \[
    t\rightarrow\int_0^t (-A)^{\delta/2} S_{t-s}B(u_s)dW_s
    \]
    admits a $\mathbb{P}$-almost surely continuous modification
and there holds
    \[
    \mathbb{E}\left[\sup_{t\leq T}\norm{\int_0^t (-A)^{\delta/2} S_{t-s}B(u_s)dW_s}^q_X\right]\leq C T^{q/2(1-\delta)}\mathbb{E}\left[\sup_{t\leq T} \norm{B(u_t)}^p_{\gamma(U,X)}\right].
    \]
    \end{theorem}
    \begin{remark}[Comparison to classical maximal inequality]
   Recall that the classical maximal inequality for stochastic convolutions in 2-smooth Banach spaces reads for $q>0$
   \[
   \mathbb{E}\left[\sup_{t\leq T} \norm{\int_0^t S_{t-s} B(u_s)dW_s}^q_X\right]\leq C\mathbb{E}\left[\left(\int_0^T \norm{B(u_s)}_{\gamma(U,X)}^2ds\right)^{q/2}\right]
   \]
(see \cite{brzezniak1997}, \cite{maxinvanneerven}), which is considerably sharper than the above inequality for the case $\delta=0$. Yet, we remark that for the fixed point argument in $Z_{q, X}$ below, the coarser inequality of Theorem \ref{max_ineq} is sufficient. For sharper maximal inequalities in the infinite time horizon case, we refer to \cite{veraar_note_2011} and \cite{van_neerven_stochastic_2012}. We also refer to the recent survey on maximal inequalities for stochastic convolutions in $2$-smooth Banach spaces \cite{neerven2020maximal}.
    \end{remark}

    \begin{remark}[The stochastic integral is well defined]
    \label{well_defined}
      Note first that due to the assumption of $A$ being the generator of an analytic semigroup on $X$, we have
   \[
   ||(-A)^\delta S_t||_{L(X)}\leq C_\delta t^{-\delta}.
   \]
Moreover, by the the ideal property of $\gamma$-radonifying operators, one has by Itô's isomorphism in 2-smooth Banach spaces for any $t\in [0,T]$
   \begin{align*}
      &\mathbb{E}\left[\norm{\int_0^t (-A)^{\delta/2}S_{t-s} B(u_s)dW_s}_X^q\right]\\
      \leq C&\mathbb{E}\left[\left(\int_0^t \norm{(-A)^{\delta/2}S_{t-s} B(u_s)}_{\gamma(U,X)}^2ds\right)^{q/2}\right]\\
      \leq C&\mathbb{E}\left[\left(\int_0^t \norm{(-A)^{\delta/2}S_{t-s}}_{L(X)}^2 \norm{B(u_s)}_{\gamma(U,X)}^2ds\right)^{q/2}\right]\\
     \leq C &\mathbb{E}\left[\left(\int_0^t \frac{1}{(t-s)^{\delta}}\norm{B(u_s)}_{\gamma(U,X)}^2ds\right)^{q/2}\right]\\
      \leq C&\mathbb{E}\left[\sup_{t\leq T} \norm{B(u_t)}^q_{\gamma(U,X)}\right]\left(\int_0^t\frac{1}{s^\delta}ds\right)^{q/2}\\
      \leq C &t^{q/2(1-\delta)}\mathbb{E}\left[1+\sup_{t\leq T} \norm{u_t}^q_X\right]\\
      \leq C& T^{q/2(1-\delta)}(1+\norm{u}_{Z_{q,X}}^q)
   \end{align*}
hence, the stochastic integral in question is well defined provided $u\in Z_{q, X}$.
    \end{remark}
    
    \begin{corollary}[Distributional regularity for $\delta\geq 1$]
    Consider the case $X=L^p(\mathbb{T}^N)$ for $p\geq 2$ and $A=\Delta-1$. Note that due to 
    \[
\left( D((-\Delta+1)^{\alpha/2}), ||(-\Delta+1)^{\alpha/2} \cdot||_{L^p(\mathbb{T}^N)}\right)\simeq \left( H^{\alpha,p}(\mathbb{T}^N), ||\cdot||_{H^{\alpha,p}(\mathbb{T}^N)}\right)
\]
 one obtains for $\alpha>(\delta-1)^+$ and $q> 2/(1-(\delta-\alpha))$, provided all other conditions of Theorem \ref{max_ineq} are met,
\begin{align*}
    &\mathbb{E}\left[\sup_{t\leq T}\norm{\int_0^t (-\Delta+1)^{\delta/2}S_{t-s} B(u_s)dW_s}_{H^{-\alpha, p}(\mathbb{T}^N)}^q\right]\\
    \leq  C&\mathbb{E}\left[\sup_{t\leq T}\norm{\int_0^t (-\Delta+1)^{(\delta-\alpha)/2}S_{t-s} B(u_s)dW_s}_{L^{ p}(\mathbb{T}^N)}^q\right]\\
    \leq C&T^{q/2(1-(\delta-\alpha))}\mathbb{E}\left[\sup_{t\leq T} \norm{B(u_t)}^p_{\gamma(U, L^p(\mathbb{T}))}\right]
\end{align*}
where we exploited that for $\alpha\geq 0$, one can pass the bounded operator $(-\Delta+1)^{-\alpha/2}$ underneath the stochastic integral. The above permits thus to conclude that for $\delta\geq 1$ the convolution process
\[
t\rightarrow \int_0^t (-\Delta+1)^{\delta/2}S_{t-s} B(u_s)dW_s
\]
takes values in $H^{-\alpha, p}(\mathbb{T}^N)$ for all $\alpha>\delta-1$, $\mathbb{P}$-almost surely.
    \end{corollary}
We now turn to the main theorem to be demonstrated in the following section.
\begin{theorem}[Mild solutions in $Z_{q, W^{m, p}(\mathbb{T}^N)}$]
\label{main_spde_result}
Let $B_1, \ldots, B_d\in C^{m}(\mathbb{T}^N\times \R)$ and $F\in C^m(\R)$ have bounded derivatives up to order $m$ satisfying
\[
\sum_{i=1}^d |B_i(x, \xi)|^2\leq C (1+|\xi|^2).
\]
Then for $\delta_0, \delta_1\in [0, 1)$, $q> 2/(1-\delta_1)$, $p\geq 2$ and $u_0\in L^q(\Omega, W^{m,p}(\mathbb{T}^N))\cap L^{mq}(\Omega, W^{1, mp}(\mathbb{T}^N))$ such that $u_0$ is $\mathcal{F}_0$-measurable, the equation
\begin{equation}
    \begin{split}
         du_t&=(\Delta -1)u_tdt+(-\Delta+1)^{\delta_0}F(u_t) dt+(-\Delta+1)^{\delta_1/2}\sum_{k=1}^dB_i(u_t)d\beta^i_t \\
    u(0)&=u_0
    \label{main_eqn1}
    \end{split}
\end{equation}
admits a unique mild solution  $u\in Z_{q, W^{m,p}(\mathbb{T}^N)}\cap Z_{mq, W^{1, mp}(\mathbb{T}^N)}$ satisfying
\begin{align*}
    &||u||_{Z_{q, W^{m,p}(\mathbb{T}^N)}}^q+||u||_{Z_{mq, W^{1, mp}(\mathbb{T}^N)}}^{mq}\\
    \leq C&\left(1+\mathbb{E}||u_0||_{W^{m,p}(\mathbb{T}^N)}^q+\mathbb{E}||u_0||_{W^{1,mp}(\mathbb{T}^N)}^{mq} \right).
\end{align*}
\end{theorem}
\begin{remark}[Strong solutions]
Note that by the Sobolev embedding theorem, the above Theorem \ref{main_spde_result} implies that the constructed mild solutions lie in $Z_{q, C^{m-1, \lambda}(\mathbb{T}^N)}$ for $\lambda\in (0,1-N/p)$. In particular, for $m\geq 3$, this permits to identify the mild solution as a strong one. 
\end{remark}
 \begin{remark}[Other types of unboundedness]
    The statement of Theorem \ref{main_spde_result} can be extended to equations of the form
    \[
    du_t=(\Delta-1) u_tdt+\mbox{div}(F(u_t))dt+(-\Delta+1)^{\delta_1/2}\sum_{i=1}^dB_i(u_t)d\beta^i_t
    \]
    for $F\in C^m(\R, \R^N)$ by interpreting them as
    \begin{equation*}
        \begin{split}
             du_t&=(\Delta-1) u_tdt+(-\Delta+1)^{1/2}\left( (-\Delta+1)^{-1/2}\mbox{div}(F(u_t))\right)dt\\
             &+(-\Delta+1)^{\delta_1/2}\sum_{i=1}^dB_i(u_t)d\beta^i_t
        \end{split}
    \end{equation*}
    and verifying that the associated
 Nemytskii operator
    \begin{align*}
        f: X&\rightarrow \gamma(U,X)\\
        u&\rightarrow (-\Delta+1)^{-1/2} \mbox{div}(F(u))
    \end{align*}
    satisfies the necessary conditions stated in Lemmas  \ref{Nemytskii}, \ref{Nemytskii_sobolev} and \ref{Nemytskii_sobolev_higher_order}. Heuristically, this is possible as the inverse root of the shifted Laplacian compensates for the unboundedness of the divergence.    For the precise statement refer to Theorem 2.1 in \cite{Hofmanova2013}.
    \end{remark}

\section{Proof of the main results}

\subsection{Proof of Theorem \ref{max_ineq}}

For the proof of the maximal inequality, we use the factorization method due to \cite{factorization}. Exploiting the identity
    \[
    \int_\sigma^t (t-s)^{\alpha-1}(s-\sigma)^{-\alpha}ds=\frac{\pi}{\sin \pi \alpha}
    \]
    which holds for any $\sigma\leq t$ in $[0,T]$ and $\alpha \in (0,1)$, we have due to Fubini's theorem (which we may apply due to Remark \ref{well_defined})
    \begin{align*}
        &\int_0^t (-A)^{\delta/2} S_{t-s} B(u_s)dW_s\\
        =&\frac{\sin \pi \alpha}{\pi}\int_0^t (t-s)^{\alpha-1}(-A)^{\delta/2} S_{t-s} \underbrace{\left(\int_0^s (s-\sigma)^{-\alpha} S_{s-\sigma} B(u_\sigma)dW_\sigma \right)}_{=:(Yu)_s}ds.
    \end{align*}
    We proceed with the proof in two steps, made up of the following two Lemmata.
    
    \begin{lemma}
    \label{factor1}
    Let $\alpha\in (0,1)$, $\delta\in [0,1)$ and $p>1$ such that
    \[
    \lambda:=\frac{p}{p-1}(1+\delta/2-\alpha)<1.
    \]
    Then the family of operators $(G_t)_t$ defined via
    \begin{align*}
        \begin{array}{rl}
        G_t: L^p([0,T], X)&\rightarrow X\\
        f &\rightarrow \int_0^t (t-s)^{\alpha-1}(-A)^{\delta/2} S_{t-s} f(s)ds
        \end{array}
    \end{align*}
    is uniformly continuous in the sense that
    \[
    \sup_{t\leq T}||G_t f||_X\leq C \left(\frac{1}{1-\lambda}\right)^{\frac{p-1}{p}}T^{\frac{p-1}{p}(1-\lambda)}||f||_{L^p([0,T], X)}.
    \]
    Moreover, for every fixed $f\in L^p([0,T], X)$ the mapping $t\rightarrow G_tf$ is continuous as a mapping from $[0,T]$ to $X$. 
    \end{lemma}
    \begin{proof}
    Note again that one has
    \[
    ||(-A)^{\delta/2} S_{t-s}||_{L(X)}\leq C (t-s)^{-\delta/2}
    \]
    and therefore by the triangle inequality for Bochner integrals and Hölder's inequality
    \begin{align*}
        ||G_tf||_X&\leq C \int_0^t  (t-s)^{\alpha-1-\delta/2} ||f(s)||_Xds\\
        &\leq C\left(\int_0^t (t-s) ^{-\lambda}ds\right)^\frac{p-1}{p} ||f||_{L^p([0,T], X)}\\
        &\leq C \left(\frac{1}{1-\lambda}\right)^{\frac{p-1}{p}} T^{(-\lambda+1)\frac{p-1}{p}}||f||_{L^p([0,T], X)}.
    \end{align*}
    Towards continuity, suppose  $f\in C([0,T], X)$, then 
    \begin{align*}
    \norm{G_tf-G_sf}_X
    &\leq \norm{\int_0^su^{\alpha-1}(-A)^{\delta/2}S_u\left(f(t-u)-f(s-u)\right)du}_X\\
    &+\norm{\int_s^tu^{1-\alpha}(-A)^{\delta/2}S_uf(t-u)du}_X\\
    &\leq C \left(\int_0^s u^{-\lambda}du\right)^{\frac{p-1}{p}}\left(\int_0^s\norm{f(t-u)-f(s-u)}_X^pdu\right)^{1/p}\\
    &+C \left(\int_s^tu^{-\lambda}du\right)^{\frac{p-1}{p}} \norm{f}_{L^p([0,T], X)}. 
\end{align*}
Due to the assumed continuity of $f$, the first expression in the above estimate vanishes as $t$ goes to $s$, which together with the continuity of the Lebesgue integral in the second expression yields continuity of $t\rightarrow G_tf$ for $f\in C([0,T], X)$. Together with the already established uniform bound on the operator family, this permits to employ a density argument, establishing continuity for $f\in L^p([0,T], X)$.
    \end{proof}
    
    \begin{lemma}
    \label{factor2}
    Suppose that $\alpha\in (0,1/2)$.
    Then the mapping
  \begin{align*}
  \begin{array}{rl}
     Y: L^p(\Omega, C([0,T], X))&\rightarrow L^p(\Omega, L^p([0,T], X))\\
      u&\rightarrow Yu
      \end{array}
  \end{align*}
defined via
 \[
 (Yu)_s:=\int_0^s (s-\sigma)^{-\alpha} S_{s-\sigma} B(u_\sigma)dW_\sigma
 \]
  satisfies
  \[
   ||Yu||_{L^p(\Omega, L^p([0,T], X))}^p\leq C_{\alpha, p} T^{p/2(1-2\alpha)+1}\mathbb{E}\left[\sup_{s\leq T} ||B(u_s)||^p_{\gamma(U,X)}\right]
  \]
  where
  \[
  C_{\alpha, p}=C \left(\frac{1}{1-2\alpha}\right)^{p/2}\frac{1}{p/2(1-2\alpha)+1}.
  \]
    \end{lemma}
    \begin{proof}
 One has
 \begin{align*}
     &||Yu||^p_{L^p(\Omega, L^p([0,T], X))}\\
     &=\mathbb{E}\left[\int_0^T \norm{\int_0^s (s-\sigma)^{-\alpha} S_{s-\sigma} B(u_\sigma)dW_\sigma}_X^pds\right]\\
     &=\int_0^T \mathbb{E}\left[\norm{\int_0^s (s-\sigma)^{-\alpha} S_{s-\sigma} B(u_\sigma)dW_\sigma}_X^p\right]ds\\
     &\leq C\int_0^T \mathbb{E}\left[\left(\int_0^s (s-\sigma)^{-2\alpha} || B(u_\sigma)||_{\gamma(U,X)}^2d\sigma\right)^{p/2}\right]ds\\
     &\leq C \mathbb{E}\left[\sup_{\sigma\leq T} ||B(u_\sigma)||_{\gamma(U,X)}^p\right]\int_0^T( \int_0^s(s-\sigma)^{-2\alpha}  d\sigma  )^{p/2}ds\\
     &\leq C_{\alpha, p} T^{p/2(1-2\alpha)+1} \mathbb{E}\left[\sup_{\sigma\leq T} ||B(u_\sigma)||_{\gamma(U,X)}^p\right]
 \end{align*}
where we crucially exploited Itô's isomorphism for $X$ a 2-smooth Banach space, i.e. 
\[
\mathbb{E}\left[||\int_0^t \phi_sdW_s||^p_X\right]\leq C\mathbb{E}\left[\left(\int_0^t ||\phi_s||_{\gamma(U,X)}^2ds\right)^{p/2}\right],
\]
the case $p=2$ corresponding to the Itô-isometry in the Hilbert space setting. This inequality is also what one obtains from the classical maximal inequality for stochastic convolutions in \cite{maxinvanneerven} by considering the trivial semigroup $S_t=\mbox{Id}$.
    \end{proof}

    \begin{remark}
   We wish to next combine the two previous Lemmata \ref{factor1} and \ref{factor2}. Note that to this end, we need to demand for $\delta\in [0,1)$ that
   \[
    \lambda:=\frac{p}{p-1}(1+\delta/2-\alpha)<1
   \]
   and $\alpha\in (0,1/2)$. Remark that the least restrictive condition on $\lambda$ is obtained by choosing $\alpha$ as large as possible, i.e. setting for some $\epsilon>0$ small $\alpha:=1/2-\epsilon/2$, we obtain 
   \[
   p>\frac{2}{1-\epsilon-\delta}
   \]
   which, since $\epsilon>0$ can be chosen arbitrarily small, resorts to demanding
   \[
    p>\frac{2}{1-\delta}
   \]
   as required in the statement of Theorem \ref{max_ineq}.
    \end{remark}
    
    Hence, assuming $p>2/(1-\delta)$, we may put together the two previous Lemmata \ref{factor1} and \ref{factor2} and we obtain
    \begin{align*}
        &\mathbb{E}\left[\sup_{t\leq T} ||\int_0^t (-A)^{\delta/2} S_{t-s} B(u_s)dW_s||^p_X\right]\\
        &\leq C T^{(-\lambda+1) (p-1)}\mathbb{E}\left[||Y||_{L_p([0,T], X)}^p\right]\\
        &=CT^{p-1-p(1+\delta/2-\alpha)}||Y||_{L^p(\Omega, L^p([0,T], X)}^p\\
        &\leq C_{\alpha, p, \delta} T^{p/2(1-\delta)}\mathbb{E}\left[\sup_{t\leq T} ||B(u_t)||^p_{\gamma(U,X)}\right]
    \end{align*}
    where
    \[
    C_{\alpha, p, \delta}=C\left(\frac{1}{1-2\alpha}\right)^{p/2}\frac{1}{p/2(1-2\alpha)+1}\left(\frac{1}{1-\lambda}\right)^{\frac{p-1}{p}}
    \]
    completing the proof to Theorem \ref{max_ineq}.
        \begin{remark}
    We stress the fact that this proof using the factorization method can not be used to recover the sharper classical maximal inequality in the case of a bounded diffusion term, i.e. $\delta=0$. Indeed, the factorization method was originally used in \cite{factorization} to prove existence of a continuous version of the stochastic convolution (among other regularity results) and not to prove a maximal inequality. As will be seen in the proof of Theorem \ref{main_spde_result} though, this coarser (yet more general) inequality is sufficient to implement a fixed point argument nonetheless.
    \end{remark}

\subsection{Proof of Theorem \ref{main_spde_result}}
Having established the maximal inequality of Theorem \ref{max_ineq}, we can now generalize the strategy employed by Hofmanová in \cite{Hofmanova2013} in order to prove Theorem \ref{main_spde_result}. This strategy consists in showing that first there exists a unique mild solution to \eqref{main_eqn} in $Z_{q, L^p(\mathbb{T}^N)}$ via a Banach fixed point argument. Considering the corresponding sequence of Picard iterations $(u^n)_n\subset Z_{q, L^p(\mathbb{T}^N)}$, one is able to show thanks to the maximal inequality of Theorem \ref{max_ineq} the sequence $(u^n)_n$ is also uniformly bounded in $Z_{q, W^{m,p}(\mathbb{T}^N)}$ (provided the initial condition lies in this space), allowing to conclude that the corresponding limit (in the topology of $Z_{q, L^p(\mathbb{T}^N)}$), actually already lies in $Z_{q, W^{m,p}(\mathbb{T}^N)}$. By the Sobolev embedding theorem, this permits to conclude that the unique mild solution is differentiable in space and hence a strong solution to \eqref{main_eqn}.\newline\indent
For the sake of readability, we choose to split up the steps mentioned above into two parts: In a first part (Lemmas denoted as abstract statements), we consider the generic setting in which $F$ and $B$ are seen as operators with suitable properties. In a second part, we recall for the convenience of the reader results of Hofmanová in \cite{Hofmanova2013} to justify why the given functions $F$ and $(B_i)_{i\leq d}$ give rise to associated Nemytskii operators with such properties (Lemmas denoted Nemytskii operator type results).

\subsubsection*{Mild solutions in $Z_{q, L^p(\mathbb{T}^N)}$}
\begin{lemma}[Abstract statement]
    For $\delta_0, \delta_1\in [0,1)$, let $q>\frac{2}{1-\delta_1}$. Let $X$ be a 2-smooth Banach space, let $W$ be a cylindrical Brownian motion on a separable Hilbert space $U$. Suppose that
  $B: X\rightarrow  \gamma(U,X)$
   and  $F:X\rightarrow X$ are Lipschitz continuous. Let $A:D(A)\subset X\rightarrow X$ be generator of an analytic contraction semigroup $(S_t)_{t\leq T}\subset L(X)$. Then for any $T<\infty$ and any $\mathcal{F}_0$-measurable $u_0\in L^q(\Omega, X)$ , the stochastic partial differential equation
   \begin{equation}
       \begin{split}
           du_t&=A u_tdt+(-A)^{\delta_0}F(u)dt+(-A)^{\delta_1/2}B(u_t)dW_t\\
        u(0)&=u_0
        \label{abstract_eqn}
       \end{split}
   \end{equation}
    admits a unique mild solution, meaning there exists a unique progressively measurable process $u\in Z_{q, X}$ such that
    \[
    u_t=S_tu_0+\int_0^t (-A)^{\delta_0} S_{t-s} F(u_s)ds+\int_0^t (-A)^{\delta_1/2} S_{t-s} B(u_s)dW_s.
    \]
Moreover $u$ satisfies
\[
||u||_{Z_{q, X}}^q\leq C(1+\mathbb{E}||u_0||^q_X).
\]
\label{abstract_statement}
    \end{lemma}

    \begin{proof}
    One uses the classical Banach fixed point theorem, i.e. a contraction argument in $Z_{q,X}$ for a sufficiently small time horizon. Consider the operator $\mathcal{K}:\  Z_{q, X} \rightarrow Z_{q, X}$, defined by
    \begin{align*}
       (\mathcal{K}u)_t:= S_{t}u_0+\int_0^{t}(-A)^{\delta_0} S_{t-s} F(u_s)ds+ \int_0^{t} (-A)^{\delta_1/2} S_{t-s} B(u_s)dW_s.
    \end{align*}
    Notice that the assumed Lipschitz continuity of $B$ and $F$ imply growth conditions of the form
    \[
    \norm{B(u)}_{\gamma(U, X)}^q\leq C(1+\norm{u}_{X}^q)
    \]
    and
    \[
    \norm{F(u)}_X^q\leq C(1+\norm{u}_X^q).
    \]

    Concerning the term associated with the drift, one has therefore
    \begin{align*}
        \mathbb{E}\left[\sup_{t\leq T} ||\int_0^t (-A)^{\delta_0} S_{t-s}F(u_s)ds||^q\right]&\leq C \mathbb{E}\left[\left(\int_0^t \frac{1}{(t-s)^{\delta_0}}||F(u_s)||_Xds\right)^q\right]\\
        &\leq C T^{q(1-\delta_0)} \mathbb{E}\left[\sup_{t\leq T} ||F(u_t)||_X^q\right]\\
        &\leq C T^{q(1-\delta_0)} (1+||u||_{Z_{q, X}}^q)
    \end{align*}
    as well as 
    \begin{align*}
        &\mathbb{E}\left[\sup_{t\leq T} ||\int_0^t (-A)^{\delta_0} S_{t-s}(F(u_s)-F(v_s))ds||^q_X\right]\\
        \leq &C  T^{q(1-\delta_0)}\mathbb{E}\left[\sup_{t\leq T} ||F(u_t)-F(v_t)||_X^q\right]\\
        \leq &CT^{q(1-\delta_0)} ||u-v||_{Z_{q, X}}^q.
    \end{align*}
    \noindent
    Concerning the term associated with the diffusion, note that because of the maximal inequality of Theorem \ref{max_ineq} one has
    \begin{align*}
      &\mathbb{E}\left[\sup_{t\leq T} ||\int_0^t (-A)^{\delta/2}S_{t-s} B(u_s)dW_s||_X^q\right]\\
       &\leq CT^{q/2(1-\delta)}\mathbb{E}\left[\sup_{t\leq T} \norm{B(u_t)}_X^q\right]\\
       &\leq CT^{q/2(1-\delta)}(1+||u||_{Z_{q,X}}).
    \end{align*}
    as well as
     \begin{align*}
        &\mathbb{E}\left[\sup_{t\leq T} ||\int_0^t (-\Delta)^{\delta/2}S_{t-s} ( B(u_s)-B(v_s))dW_s||_X^q\right]\\
        \leq& C  T^{q/2(1-\delta)} \mathbb{E}\left[\sup_{t\leq t} ||B(u_t)-B(v_t)||_{\gamma(U, X)}^q\right]\\
        \leq& C T^{q/2(1-\delta)}\mathbb{E}\left[\sup_{t\leq T} ||u_t-v_t||_X^q\right]\\
        =&C T^{q/2(1-\delta)} ||u-v||^q_{Z_{q, X}}
    \end{align*}
   due to the assumed Lipschitz continuity of $B$. Overall, one concludes that 
    \begin{align}
       || \mathcal{K}(u)||_{Z_{q, X}}^q\leq  C\left(||u_0||_{Z_q}^q+(T^{q(1-\delta_0)}+T^{q/2(1-\delta_1)})(1+||u||_{Z_{q, X}}^q)\right)
       \label{uniform_picard}
    \end{align}
   meaning that $\mathcal{K}$ maps $Z_{q, X}$ into itself. Moreover, there holds
    \begin{align*}
        ||\mathcal{K}(u)-\mathcal{K}(v)||_{Z_{q, X}}^q
        &\leq C\left(T^{q(1-\delta_0)}+T^{q/2(1-\delta_1)}\right)||u-v||_{Z_{q, X}}^q.
    \end{align*}
    By choosing $T$ sufficiently small, $\mathcal{K}$ is a contraction on $Z_{q, X}$ and hence admits a unique fixed point that by definition coincides with a mild solution.\newline\indent
 Continuity in time is a consequence of the existence of a continuous modification of the stochastic convolution in Theorem \ref{max_ineq} as well as the continuity of the Bochner integral and of the semigroup. This permits to compute the unique solution to the equation in question on $[T, 2T]$ with initial condition $u_T$ calculated previously, etc., thus recovering existence and uniqueness of mild solutions on arbitrary
 finite time horizons.\newline\indent
 Finally, the bound on solutions is derived from weak-*-lower semicontinuity of the norm $||\cdot||_{Z_{q, X}}$, the strong convergence of Picard iterations and the above estimate \eqref{uniform_picard}.
 \end{proof}

\begin{lemma}(A Nemytskii operator type result for $L^p(\mathbb{T}^N)$)
\label{Nemytskii}
Let $U$ be a $d$-dimensional Hilbert space with orthonormal basis $(e_i)_{i\leq d}$. Let $B_1, \ldots, B_d\in C^1(\mathbb{T}^N\times \R; \R)$ have bounded derivative and satisfy the growth condition
\[
\sum_{i=1}^d|B_i(x, \xi)|^2\leq C(1+|\xi|^2).
\]
Then the associated Nemytskii operator
\begin{align*}
    B: L^p(\mathbb{T}^N)&\rightarrow \gamma(U, L^p(\mathbb{T}^N))\\
    z&\rightarrow \left(u\rightarrow \sum_{i=1}^dB_i(\cdot, z(\cdot))\langle u, e_i\rangle\right)
\end{align*}
is well defined and Lipschitz continuous.
 Suppose $F\in C^1(\R)$ is of bounded derivative. Then the associated Nemytskii operator
\begin{align*}
    F: L^p(\mathbb{T}^N) &\rightarrow L^p(\mathbb{T}^N)\\
    u&\rightarrow F(u)
\end{align*}
is well defined and Lipschitz continuous.
\end{lemma}
\begin{proof}
See Proposition 4.1 in \cite{Hofmanova2013}.
\end{proof}

\begin{corollary}
Suppose the conditions of Theorem \ref{main_spde_result} are satisfied. Then there exists a unique mild solution $u\in Z_{q, L^p(\mathbb{T}^N)}$ to the problem \eqref{main_eqn1}.
\end{corollary}

\subsubsection*{Mild solutions in $Z_{q, W^{1, p}(\mathbb{T}^N)}$}

\begin{lemma}[Abstract statement]
\label{abstract_statement_refined}
  Suppose all conditions of Lemma \ref{abstract_statement} are met. Let $\tilde{X}$ be another 2-smooth Banach space, continuously embedded into $X$, the embedding operator being the identity and suppose $u_0\in L^q(\Omega, \tilde{X})$. Suppose that $B$ seen as an operator $B:\tilde{X}\rightarrow \gamma(U, \tilde{X})$ is well defined and satisfies the growth condition 
  \[
  ||B(u)||_{\gamma(U, \tilde{X})}^q\leq C(1+||u||_{\tilde{X}}^q).
  \]
  Suppose that $F$ seen as an operator $F:\tilde{X}\rightarrow \tilde{X}$ is well defined satisfying the growth condition
  \[
  ||F(u)||_{\tilde{X}}^q\leq C(1+||u||^q_{\tilde{X}}). 
  \]
  Suppose moreover that $A|_{\tilde{X}}$ generates a strongly continuous contraction semigroup $(\tilde{S}_t)_t\subset L(\tilde{X})$ such that $S_t|_{\tilde{X}}=\tilde{S}_t$. 
  Then the unique mild solution $u\in Z_{q, X}$ to \eqref{abstract_eqn} is also the unique mild solution $u\in Z_{q,\tilde{X}}$ to \eqref{abstract_eqn} satisfying
  \[
  ||u||^q_{Z_{q, \tilde{X}}}\leq C(1+\mathbb{E}||u_0||_{\tilde{X}}^q).
  \]
  \end{lemma}
  \begin{proof}
   Since the conditions of the previous theorem are met, there exists a unique $u\in Z_{q, X}$ which is the strong limit of Picard iterations, given via the recursive formula $u^{0}=u_0$ and
   \[
   u^{n}_t=\underbrace{S_tu_0+\int_0^t (-A)^{\delta_0}S_{t-s}F(u_s^{n-1})ds+\int_0^t (-A)^{-\delta_1/2} S_{t-s} B(u^{n-1}_s)dW_s}_{=:\mathcal{K}(u^{n-1})}
   \]
   for $n\geq 1$. Note that due to the maximal inequality of Theorem \ref{max_ineq} for the stochastic integral and the triangle inequality for the Bochner integral, one obtains an estimate similar to \eqref{uniform_picard} but in the space $Z_{q, \tilde{X}}$ namely
   \begin{align*}
     ||\mathcal{K}(u^{n})||_{Z_{q,\tilde{X}}}^q
     \leq \left(||u_0||_{Z_{q, \tilde{X}}}^q+(T^{q(1-\delta_0)}+ T^{q/2(1-\delta_1)})(1+ ||u^{n-1}_t||^q_{Z_{q\tilde{X}}})\right)
   \end{align*}
We conclude recursively that 
\begin{align*}
    ||u^{n}||_{Z_X}^p&\leq C(1+ ||u_0||_{Z_{q, \tilde{X}}}^q)\sum_{k=0}^{n-1} (CT^{q/2(1-\delta_1)})^k+(CT^{q(1-\delta_0)})^k.
\end{align*}
For $T$ sufficiently small the above geometric series converges and one obtains a uniform bound on the sequence of Picard iterations $(u^n)_n$ in the the space $Z_{q, \tilde{X}}$. By Alaoglu's theorem, one can extract a weak-*-convergent subsequence with limit $v\in Z_{\tilde{X}}$. Since $Z_{\tilde{X}}\hookrightarrow Z_X$, one also has
\[
u^n\overset{*}{\rightharpoonup} v
\]
in $Z_X$ and by uniqueness of limits, $v=u$, meaning the limit of Picard iterations already lies in $Z_{\tilde{X}}$. By Lemma \ref{Banach_space_consistency}, one can identify the stochastic integral in $X$ with the stochastic integral in $\tilde{X}$, meaning in the notation introduced in Lemma \ref{Banach_space_consistency} below
\begin{align*}
    u_t&=S_tu_0+\int_0^t (-A)^{\delta_0}S_{t-s} F(u_s)ds+\left(X\int_0^t\right) (-A)^{-\delta_1/2} S_{t-s}B(u_s)dW_s\\
    &=\tilde{S}_tu_0+\int_0^t \tilde{S}_{t-s} (-A)^{\delta_0}F(u_s)ds+\left(\tilde{X}\int_0^t\right) (-A)^{-\delta_1/2} \tilde{S}_{t-s}B(u_s)dW_s
\end{align*}
which is the definition of $u$ being a mild solution in $Z_{q, \tilde{X}}$. Uniqueness follows from the embedding $Z_{\tilde{X}}\hookrightarrow Z_X$. Continuity follows from the continuous version of the modified stochastic convolution thanks to Theorem \ref{max_ineq}. The estimate of the solutions follows from the weak-*-lower semicontinuity of the norm and the uniform bound on Picard iterations established. 
  \end{proof}

 \begin{lemma}(A Nemytskii operator type result for $W^{1,p}(\mathbb{T}^N)$)
\label{Nemytskii_sobolev}
Let $U$ be a $d$-dimensional Hilbert space with orthonormal basis $(e_i)_{i\leq d}$. Let $B_1, \ldots, B_d\in C^1(\mathbb{T}^N\times \R)$ be differentiable and of bounded derivative. 
Then the associated Nemytskii operator
\begin{align*}
  B: W^{1,p}(\mathbb{T}^N)&\rightarrow \gamma(U, W^{1,p}(\mathbb{T}^N))\\
    z&\rightarrow \left(u\rightarrow \sum_{i=1}^dB_i(\cdot, z(\cdot))\langle u, e_i\rangle\right)
\end{align*}
is well defined and satisfies the linear growth condition
\[
||B(z)||_{\gamma(U, W^{1,p}(\mathbb{T}^N))}\leq C(1+||z||_{W^{1,p}(\mathbb{T}^N)})
\]
for all $p\geq 1$. Suppose $F\in C^1(\R)$ has bounded derivative, then the operator
\begin{align*}
    W^{1, p}(\mathbb{T}^N)&\rightarrow W^{1, p}(\mathbb{T}^N)\\
    u &\rightarrow F(u)
\end{align*}
is well defined and satisfies 
\[
||F(u)||_{W^{1, p}(\mathbb{T}^N)}\leq C(1+||u||_{W^{1, p}(\mathbb{T}^N)})
\]
\end{lemma}
\begin{proof}
See Proposition 4.2 in \cite{Hofmanova2013} and Lemma \ref{Nemytskii_simple}.
\end{proof}

\begin{corollary}
Suppose the conditions of Theorem \ref{main_spde_result}. Then there exists a unique mild solution $u\in Z_{q, W^{1,p}(\mathbb{T}^N)}$ to the problem \eqref{main_eqn1}.
\end{corollary}

\subsubsection*{Mild solutions in $Z_{q, W^{m,p}(\mathbb{T}^N)}$}

Constructing mild solutions in $Z_{q, W^{1,p}(\mathbb{T}^N)}$ essentially relied on the fact that Nemytskii operators associated to the $C^1(\mathbb{T}^N\times \R)$  functions $B_1, \ldots, B_d$ and the $C^1(\R)$ function $F$ are of linear growth thanks to Lemma \ref{Nemytskii_simple}. This reasoning breaks down when considering Sobolev spaces of higher order (as required in order to be able to implement the Sobolev embedding theorem). Nontheless, one is able to recover (polynomial) growth conditions in this setting, which will turn out to be sufficient to again control Picard iterations in $Z_{q, W^{m,p}(\mathbb{T}^N)}$.

\begin{lemma}(Abstract statement)
\label{abstract_statement_2}
Suppose all conditions of Lemma  \ref{abstract_statement} are satisfied. Let $X_1, X_2$ be 2-smooth Banach spaces such that $X_1\hookrightarrow X$ the embedding operator being the identity and suppose $u_0\in L^q(\Omega, X_1)\cap L^{mq}(\Omega, X_2)$ for some $m\geq 1$. Suppose that $B$ seen as an operator $B:X_1\cap X_2\rightarrow \gamma(U, X_i)$ is well defined for both $i=1, 2$ and satisfies the growth conditions
\[
||B(u)||^q_{\gamma(U, X_1)}\leq C(1+||u||_{X_1}^{q}+||u||_{X_2}^{mq})
\]
and
\[
||B(u)||^q_{\gamma(U, X_2)}\leq C(1+||u||_{X_2}^q)
\]
for $u\in X_1\cap X_2$. 
Suppose that $F$ seen as an operator $F: X_1\cap X_2\rightarrow X_i$ is well defined for both $i=1,2$ and satisfies the growth conditions
\[
||F(u)||_{X_1}^q\leq C(1+||u||_{X_1}^q+||u||^{mq}_{X_2})
\]
and
\[
||F(u)||_{X_2}^q\leq C(1+||u||_{X_2}^q)
\]
for $u\in X_1\cap X_2$. Suppose moreover that $A|_{X_i}$ generates a strongly continuous contraction semigroup $(S^i_t)_t\subset L(X_i)$ such that $S_t|_{X_i}=S^i_t$ for both $i=1, 2$. Then the unique mild solution $u\in Z_{q, X}$ of Lemma \ref{abstract_statement} lies in $Z_{q, X_1}\cap Z_{mq, X_2}$, and is a continuous in time mild solution in $Z_{q, X_1}$.
\end{lemma}
\begin{proof}
Due to Lemma \ref{abstract_statement}, there exists a unique continuous solution in $Z_{q, X}$. Thanks to the growth condition imposed, one can check that $\mathcal{K}$ maps the space $Z_{q, X_1}\cap Z_{mq, X_2}$ onto itself, meaning the sequence of Picard iterations lies in $Z_{q, X_2}\cap Z_{mq, X_2}$. Due to the growth condition imposed on the Nemytskii operators, we obtain again a uniform bound on this sequence of Picard iterations in both $Z_{q, X_1}$ and $Z_{mq, X_2}$ for $T$ sufficiently small. This means there exist a weak* convergent subsequence of Picard iterations in $Z_{q, X_1}$ and a weak* convergent subsequence in $Z_{mq, X_2}$. Due to the same reasoning as in Lemma \ref{abstract_statement_refined}, both weak* limits need to coincide with the fixed point $u\in Z_{q, X}$ obtained from the application of Lemma \ref{abstract_statement}, i.e. the solution lies in $Z_{q, X_1}\cap Z_{mq, X_2}$. Exploiting Lemma \ref{Banach_space_consistency} we deduce that $u$ is a mild solution in $Z_{q, X_1}$ thanks to the continuous embedding $X_1\hookrightarrow X$. Continuity follows again from Theorem \ref{max_ineq}.
\end{proof}
\begin{lemma}(A Nemytskii operator type result for $W^{m,p}(\mathbb{T}^N)$)
\label{Nemytskii_sobolev_higher_order}
Let $U$ be a $d$-dimensional Hilbert space with orthonormal basis $(e_i)_{i\leq d}$. Let $B_1, \ldots, B_d\in C^m(\mathbb{T}^N\times \R)$ with bounded $m$-th derivative. 
Then the associated Nemytskii operator
\begin{align*}
   B: W^{1,mp}(\mathbb{T}^N)\cap W^{m,p}(\mathbb{T}^N)&\rightarrow \gamma(U, W^{m,p}(\mathbb{T}^N))\\
    z&\rightarrow \left(u\rightarrow \sum_{i=1}^dB_i(\cdot, z(\cdot))\langle u, e_i\rangle\right)
\end{align*}
is well defined and satisfies the  growth condition
\[
||B(z)||_{\gamma(U, W^{m,p}(\mathbb{T}^N))}\leq C(1+||z||_{W^{1,mp}(\mathbb{T}^N)}^m+||z||_{W^{m,p}(\mathbb{T}^N)})
\]
for all $p\geq 2$. Suppose $F\in C^m(\R)$ is of bounded $m$-th derivative. Then the associated Nemytskii operator
\begin{align*}
    F: W^{1, mp}(\mathbb{T}^N)\cap W^{m, p}(\mathbb{T}^N)&\rightarrow W^{m,p}(\mathbb{T}^N)\\
    u&\rightarrow F(u)
\end{align*}
is well defined and satisfies the growth condition
\[
||F(u)||_{W^{m,p}(\mathbb{T}^N)}\leq C\left( 1+||u||_{W^{m,p}(\mathbb{T}^N)}+||u||^m_{W^{1, mp}(\mathbb{T}^N)}\right)
\]
\end{lemma}
\begin{proof}
See Proposition 4.3 in \cite{Hofmanova2013} and Lemma \ref{Nemytskii_simple}.
\end{proof}
\begin{remark}
The preceding Lemma \ref{abstract_statement_2} whose proof is essentially based on Lemma \ref{Nemytskii_simple} conveys the reason why one has to leave the Hilbert space framework of stochastic integration and rather work in the setting of stochastic integration with values in 2-smooth Banach spaces: Since one needs to control norms in $W^{1,mp}(\mathbb{T}^N)$ and $W^{m,p}(\mathbb{T}^N)$ one leaves the Hilbert space setting $p=2$ as soon as one intends to consider orders higher than $m=1$. Note moreover that the preceding Lemma implies Lemma \ref{Nemytskii_sobolev}.
\end{remark}
Combing the results of this subsection, we conclude that under the conditions of Theorem \ref{main_spde_result} there exists a unique mild solution $u\in Z_{q, W^{m,p}(\mathbb{T}^N)}\cap Z_{mq, W^{1,mp}(\mathbb{T}^N)}$ to the problem \eqref{main_eqn1}. The bound stated in Theorem \ref{main_spde_result} follows again from weak-*-lower semicontinuity of the norms. In particular, by the Sobolev embedding theorem, this implies also that $u$ is a strong solution for $m\geq 3$.

\begin{remark}
Note that while the abstract existence and uniqueness statements of this section are formulated for general $U$-cylindrical Brownian motion, i.e. hold in the infinite dimensional setting, Lemmas \ref{Nemytskii}, \ref{Nemytskii_sobolev}, \ref{Nemytskii_sobolev_higher_order} concerning Nemytskii operator results crucially rely on $U$ being finite dimensional. 
\end{remark}

    \section{The critical equation $\delta_1=1$}
    We wish to study the critical problem\footnotemark\footnotetext{For easier reading, we suppress the additional nonlinear drift. }
   \begin{align}
       \begin{cases}
       du_t&= (\Delta-1) u_tdt+\mu(-\Delta+1)^{1/2}\sum_{i=1}^dB_i(u_t)d\beta^i_t\\
       u(0)&=u_0
       \label{ciritical_eqn}
       \end{cases}
   \end{align}
   which we will also occasionally write more compactly as 
    \begin{align*}
       \begin{cases}
       du_t&= (\Delta-1) u_tdt+\mu(-\Delta+1)^{1/2}B(u_t)dW_t\\
       u(0)&=u_0
       \end{cases}
   \end{align*}
 through approximation. 
   For $\delta\in [0,1)$ let $u^\delta$ be the unique mild solution to
   \begin{align*}
       \begin{cases}
       du_t&= (\Delta-1) u_tdt+\mu(-\Delta+1)^{\delta/2}\sum_{i=1}^dB_i(u_t)d\beta^i_t\\
       u(0)&=u_0
       \end{cases}
   \end{align*}
    in $Z_{q, W^{1, p}(\mathbb{T}^N)}$ where $\mu \in \R$ and $B_1, \ldots, B_d\in C^1(\mathbb{T}^N)$ have bounded derivative satisfying the growth condition stated in Theorem \ref{main_spde_result}.  In particular $(u^\delta)_{\delta\in [0,1)}\subset L^2(\Omega\times[0,T], W^{1,2}(\mathbb{T}^N))$. We show that for $\mu$ sufficiently small, $(u^\delta)_{\delta\in [0,1)}$ is uniformly bounded in this space. Together with a second uniform bound in a space of higher time and lower space regularity, this allows us to apply the stochastic compactness method due to Flandoli and Gatarek (refer to \cite{flandoli1995martingale} for the original article and the lecture notes \cite{hof_lec_notes} for a pedagogical introduction). Upon the establishment of uniform bounds in the aforementioned spaces, one is able to establish tightness of the laws of $(u^\delta)_\delta$, which in combination with Prokhorov's theorem and the Skorokhod representation theorem permits to conclude the existence of a martingale solution in the following sense.
    
    \begin{definition}
    We say the problem \eqref{ciritical_eqn} admits a martingale solution if there exists a filtered probability space $(\Omega', \mathcal{F}', (\mathcal{F}'_t)_t, \mathbb{P}')$, a $d$-dimensional $(\mathcal{F}'_t)_t$ Brownian motion $W'$ and a progressively measurable process $u:\Omega'\times [0,T]\to L^2(\mathbb{T}^N)$ such that  for some $\alpha>0$ we have $\mathbb{P}'$-almost surely
    \[
    u'\in C([0,T], H^{-\alpha}(\mathbb{T}^N))\cap L^2([0,T], W^{1,2}(\mathbb{T}^N))
    \]
    as well as for any $\varphi\in H^\alpha(\mathbb{T}^N)$
    \begin{equation*}
        \begin{split}
            \langle u_t, \varphi&=\langle u_0, \varphi\rangle+\int_0^t \langle (\Delta-1) u_s, \varphi\rangle ds +\int_0^t \langle (-\Delta+1)^{1/2}B(u_s), \varphi\rangle dW'_s.
        \end{split}
    \end{equation*}
    \end{definition}
    
    \bigskip
    
    \subsection{A first a priori bound}
    \begin{lemma}
There exists $\mu_0 \in \R$ such that for all $\mu^2 <\mu_0^2$, the family $(u^\delta)_{\delta\in [0,1)}$ is uniformly bounded in $L^2(\Omega\times[0,T], W^{1,2}(\mathbb{T}^N))$.
\label{a_priori_1}
\end{lemma}
\begin{proof}
By Theorem \ref{main_spde_result}, solutions $u^\delta$ take values in $Z_{q, W^{1, 2}(\mathbb{T}^N)}$ for $q>2/(1-\delta)$.
Applying Itô's formula in the Hilbert space $L^2(\mathbb{T}^N)$, one obtains
    \begin{align*}
    ||u_t^\delta||^2_{L^2(\mathbb{T}^N)}&=||u_0||^2_{L^2(\mathbb{T}^N)}-2\int_0^t \int_{\mathbb{T}^N}(|u^\delta_s|^2+|\nabla u^\delta_s|^2) dxds+\mbox{Martingale}\\
    &+\mu^2 \int_0^t\sum_{i=1}^d \int_{\mathbb{T}^N}| (-\Delta+1)^{\delta/2}B_i(u_s^\delta)|^2dxds
    \end{align*}
    \noindent
Note that since 
\[
||(-\Delta+1)^{\delta/2}v||_{L^2(\mathbb{T}^N)}\leq ||(-\Delta+1)^{1/2}v||_{L^2(\mathbb{T}^N)}
\]
for $v \in D((-\Delta+1)^{1/2})$ 
we have
\begin{align*}
    \mathbb{E}[||u_t^\delta||^2_{L^2(\mathbb{T}^N)}]&\leq\mathbb{E}[ ||u_0||^2_{L^2(\mathbb{T}^N)}]-2\mathbb{E}[\int_0^t (\norm{u^\delta_s}_{L^2(\mathbb{T}^N)}^2+||\nabla u^\delta_s||_{L^2(\mathbb{T}^N)}^2)ds]\\
    &+\mu^2 \mathbb{E}[\int_0^t || \sum_{i=1}^d(-\Delta+1)^{1/2}B_i(u_s^\delta)||_{L^2(\mathbb{T}^N)}^2ds]
\end{align*}
Note that we have moreover
\[
\left( D((-\Delta+1)^{1/2}), ||(-\Delta+1)^{1/2} \cdot||_{L^p(\mathbb{T}^N)}\right)\simeq \left( H^{1,2}(\mathbb{T}^N), ||\cdot||_{H^{1,2}(\mathbb{T}^N)}\right).
\]
Defining the Nemyskii operator 
\begin{align*}
    b: W^{1,2}(\mathbb{T}^N) &\rightarrow W^{1,2}(\mathbb{T}^N)\\
    u &\rightarrow \sum_{i=d}^dB_i(u),
\end{align*}
Lemma \ref{Nemytskii_simple} implies the growth condition 
\[
||b(u)||_{W^{1,2}(\mathbb{T}^N)}^2\leq C (1+||u||_{W^{1,2}(\mathbb{T}^N)}^2).
\]
Combining these considerations, we obtain
\begin{align*}
    || (-\Delta+1)^{1/2}\sum_{i=1}^dB_i(u_s^\delta)||_{L^2(\mathbb{T}^N)}^2&\leq C|| \sum_{i=1}^dB_i(u_s^\delta)||_{H^{1,2}(\mathbb{T}^N)}^2=||b(u_s^\delta)||^2_{W^{1,2}(\mathbb{T}^N)}\\
    &\leq C(1+||u_s^\delta||^2_{W^{1,2}(\mathbb{T}^N)})\\
    &= C\left(1+\norm{u^\delta_s}_{L^2(\mathbb{T}^N)}^2+||\nabla u_s^\delta||^2_{L^{2}(\mathbb{T}^N)}\right).
\end{align*}{}
Returning to the above application of Itô's formula, one has
\begin{align*}
    &\mathbb{E}[||u_t^\delta||^2_{L^2(\mathbb{T}^N)}]\\
    \leq&\mathbb{E}[ ||u_0||^2_{L^2(\mathbb{T}^N)}]+(C\mu^2-2)\mathbb{E}[\int_0^t (1+\norm{u^\delta_s}_{L^2(\mathbb{T}^N)}^2+||\nabla u_s^\delta||_{L^2(\mathbb{T}^N)}^2)ds]
\end{align*}{}
and thus in particular for $\mu_0=\sqrt{2/C}$ the result follows.
\end{proof}
\begin{remark}
We stress that in the previous proof, we exploited that $u^\delta$ takes values in the space $W^{1, 2}(\mathbb{T}^N)$, providing a uniform bound in $L^2(\Omega\times [0,T]\times W^{1,2}(\mathbb{T}^N))$. As seen in Theorem \ref{main_spde_result}, space regularity of $u^\delta$ can be improved to  $W^{m, 2}(\mathbb{T}^N)$ provided that the functions $(B_i)_{i\leq d}$ and the initial condition are sufficiently regular. It would be tempting to exploit this information to derive a uniform bound of $(u^\delta)_\delta$ in $L^2(\Omega\times [0,T]\times W^{m-1, 2}(\mathbb{T}^N))$ for $m\geq 2$. Note however that similarly to the third step to the proof of Theorem \ref{main_spde_result},  the operator
\begin{align*}
    b_m: W^{m,2}(\mathbb{T}^N) &\rightarrow W^{m,2}(\mathbb{T}^N)\\
    u &\rightarrow \sum_{i=d}^dB_i(u)
\end{align*}
is well defined only for $m=1$. Referring to Lemma \ref{Nemytskii_simple}, one might instead consider the operator 
\begin{align*}
    \tilde{b}_m: W^{m,2}(\mathbb{T}^N)\cap W^{1,2m}(\mathbb{T}^N)  &\rightarrow W^{m,2}(\mathbb{T}^N)\cap W^{1,2m}(\mathbb{T}^N)\\
    u &\rightarrow \sum_{i=d}^dB_i(u)
\end{align*}
with its corresponding growth condition. Notice however that this bound doesn't allow for a control of the unbounded diffusion by the drift as the norms of two different Sobolev spaces appear, preventing the establishment of uniform bounds in $L^2(\Omega\times [0,T]\times W^{m-1, 2}(\mathbb{T}^N))$ for $m\geq 2$ by a generalization of the proof strategy. 
\end{remark}
    
    \subsection{A second a priori bound}

    We begin by recalling the following Lemma due to \cite{flandoli1995martingale}.
    \begin{lemma}
    Let $p\geq 2$, $\alpha<1/2$. Let $U, H$ be separable Hilbert spaces and $W$ a $U$-cylindrical Brownian motion. Then for any progressively measurable process
    \[
    f\in L^p(\Omega\times [0,T], L_2(U, H))
    \]
    it holds that 
    \[
    I(f):=\int_0^{(\cdot)}f_sdW_s\in L^p(\Omega, W^{\alpha,p}([0,T], H))
    \]
    and there exists a constant $C=C(\alpha,p)>0$ independent of $f$ such that 
    \[
    \mathbb{E}\left[\norm{I(f)}^p_{W^{\alpha, p}([0,T], H)}\right]\leq C \mathbb{E}\left[\int_0^T\norm{f_s}^p_{L_2(U,H)}ds\right].
    \]
    \label{flandoli_frac}
    \end{lemma}
    While we can not apply this Lemma in the above context for $H=L^2(\mathbb{T}^N)$ due to the unbounded operator in front of the diffusion term, we can still obtain a bound in a suitable Bessel potential space of distributions, which will turn out to be sufficient for the stochastic compactness method.   
    
    \begin{lemma}
    Let $u\in L^2(\Omega\times [0,T], L^2(\mathbb{T}^N))$ be a progressively measurable process. Then for $B_1, \ldots, B_d\in C^1(\mathbb{T}^N)$ having bounded derivatives,    we have that 
    \[
   I(u):= \int_0^{(\cdot)}(-\Delta+1)^{\delta/2} \sum_{i=1}^d B_i(u_s)d\beta^i_s\in L^2(\Omega, W^{\alpha, 2}([0,T], H^{-1}(\mathbb{T}^N)))
    \]
    and moreover
     \[
     \mathbb{E}\left[\norm{I(u)}^2_{W^{\alpha, 2}([0,T], H^{-1}(\mathbb{T}^N))}\right]\leq C \mathbb{E}\left[\int_0^T\norm{u_s}^2_{L^2(\mathbb{T}^N)}ds\right]
    \]
    \label{flandoli_frac2}
    \end{lemma}
    \begin{proof}
    The proof follows immediately from the Nemytskii operator result of Lemma \ref{Nemytskii}, the previous Lemma \ref{flandoli_frac} as well as the fact that the family of operators $(D_\delta)_{\delta\in [0,1]}$ defined via
  \begin{equation*}
      \begin{split}
          D_\delta: L^2(\mathbb{T}^N)&\to H^{-1}(\mathbb{T}^N)\\
          v&\to (-\Delta+1)^{\delta/2}v
      \end{split}
  \end{equation*}
  is uniformly bounded and the ideal property of Hilbert-Schmidt operators.
    \end{proof}
    In combination with the first a priori bound of Lemma \ref{a_priori_1}, this permits the derivation of a second one.
    \begin{lemma}
The family $(u^\delta)_\delta$ is uniformly bounded in $ L^2(\Omega, W^{\alpha, 2}([0,T], H^{-1}(\mathbb{T}^N)))$. 
\label{a_priori_2}
    \end{lemma}
    \begin{proof}
    As $u^\delta$ is strong solution to 
    \[
    u_t^\delta=u_0+\int_0^t (-\Delta+1)u^\delta_sds+\mu \int_0^t (-\Delta+1)^{\delta/2}\sum_{i=1}^dB_i(u^\delta_s)d\beta^i_s
    \]
    we need to bound the three terms on the right hand side.        It is immediate that 
    \[
    \mathbb{E}\left[ \norm{\int_0^{(\cdot)} (-\Delta+1)u^\delta_sds }^2_{W^{1,2}([0,T], H^{-1}(\mathbb{T}^N))}\right]\leq C\mathbb{E}\left[\norm{u^\delta}^2_{L^2([0,T], H^{1}(\mathbb{T}^N))}\right]<\infty.
    \]
Moreover, by Lemma \ref{flandoli_frac2}
    \begin{equation*}
        \begin{split}
            &\mathbb{E}\left[ \norm{\int_0^{(\cdot)}(-\Delta+1)^{\delta/2} \sum_{i=1}^d B_i(u^\delta_s)d\beta^i_s}_{W^{\alpha, 2}([0,T], H^{-1}(\mathbb{T}^N)}^2\right]\\
            \leq& C \mathbb{E}[\norm{u^\delta}_{L^2([0,T]\times \mathbb{T}^N)}^2]<\infty
        \end{split}
    \end{equation*}
    Overall, this shows that indeed,
    \[
    \mathbb{E}\left[ \norm{u^\delta}^2_{W^{\alpha, 2}([0,T], H^{-1}(\mathbb{T}^N))}\right]<\infty.
    \]
        \end{proof}
        \subsection{Tightness}
        We next show that the laws of $(u^\delta)_\delta$ are tight in $L^2([0,T]\times \mathbb{T}^N)\cap C([0,T], H^{-(1+\epsilon)}(\mathbb{T}^N))$. In order to do so, we use the following result, which is central to the argument in \cite{flandoli1995martingale}.
        \begin{theorem}
        Let $B_0, B, B_1$ be Banach spaces satisfying the following embedding property
        \[
        B_0\hookrightarrow \hookrightarrow B\hookrightarrow B_1,
        \]
        where by $B_0\hookrightarrow\hookrightarrow B$ we understand that the embedding $B_0\hookrightarrow B$ is compact. Then for $p\in (1, \infty)$ and $\alpha\in (0,1)$, one has the following compact embedding
        \[
        L^p([0,T], B_0)\cap W^{\alpha, p}([0,T], B_1)\hookrightarrow\hookrightarrow L^p([0,T], B).
        \]
        Suppose moreover that 
        \[
        \alpha p>1,
        \]
        then we also have the compact embedding
        \[
        W^{\alpha, p}([0,T], B_0)\hookrightarrow\hookrightarrow C([0,T], B).
        \]
        \label{embeddings}
        \end{theorem}
        \begin{lemma}
        The laws of $(u^\delta)_{\delta\in [0,1)}$ are tight in $L^2([0,T]\times \mathbb{T}^N)\cap C([0,T], H^{-(1+\epsilon)}(\mathbb{T}^N))$ for any $\epsilon>0$.
        \end{lemma}
        \begin{proof}
        Note that with $B_0=W^{1,p}(\mathbb{T}^N)$, $B=L^2(\mathbb{T}^N)$ and $B_1=H^{-1}(\mathbb{T}^N)$, we can apply the previous Theorem \ref{embeddings} to conclude that we have the compact embedding
        \[
        Y:=L^2([0,T], W^{1,2}(\mathbb{T}^N))\cap W^{\alpha, 2}([0,T], H^{-1}(\mathbb{T}^N))\hookrightarrow\hookrightarrow L^2([0,T]\times \mathbb{T}^N).
        \]
        We conclude that for any $R>0$, we set
        \[
        B_R:=\{ u\in  Y\ |\ \norm{u}_{  L^2([0,T], W^{1,2}(\mathbb{T}^N))}+\norm{u}_{W^{\alpha, 2}([0,T], H^{-1}(\mathbb{T}^N))}\leq R\}
        \]
        is compact in $L^2([0,T]\times \mathbb{T}^N)$. Hence by Markov's inequality and the a priori bounds established in Lemmas \ref{a_priori_1} and \ref{a_priori_2} we obtain
        \begin{equation*}
            \begin{split}
     \mathbb{P}(u^\delta\in B_R^c)&\leq \mathbb{P}(\norm{u^\delta}_{ L^2([0,T], W^{1,2}(\mathbb{T}^N))}>R/2)\\
     &+\mathbb{P}(\norm{u^\delta}_{W^{\alpha, 2}([0,T], H^{-1}(\mathbb{T}^N))}>R/2)\\
     &\leq\frac{4}{R^2} \mathbb{E}\left[\norm{u^\delta}_{ L^2([0,T], W^{1,2}(\mathbb{T}^N))}^2+\norm{u^\delta}_{W^{\alpha, 2}([0,T], H^{-1}(\mathbb{T}^N))}^2\right]\\
     &\leq \frac{4C}{R^2}
            \end{split}
        \end{equation*}
        from which we conclude that $(u^\delta)_{\delta\in [0,1)}$ is tight in $L^2([0,T]\times \mathbb{T}^N)$. Moreover, notice that due to the second compact embedding stated in the previous Theorem \ref{embeddings}, we also obtain similarly tightness in $C([0,T], H^{-(1+\epsilon)}(\mathbb{T}^N))$ for any $\epsilon>0$.
        \end{proof}
        \subsection{Prokhorov, Skorokhod and Martingale representation theorem}
        Consider the sequence $(u^\delta)_{\delta\in [0,1)\cap \mathbb{Q}}$. By Prokhorov's theorem, we can find a subsequence we shall also note $(u^\delta)_{\delta\in [0,1)\cap \mathbb{Q}}$ whose laws converge weakly in $L^2([0,T]\times \mathbb{T}^N)\cap C([0,T], H^{-(1+\epsilon)}(\mathbb{T}^N))$. By Skorokhod's representation theorem, we then find a stochastic basis $(\tilde{\Omega}, \tilde{\mathcal{F}}, (\tilde{\mathcal{F}}_t)_t, \tilde{\mathbb{P}}$) as well as random variables $\tilde{u}, (\tilde{u}^\delta)_\delta$ on this stochastic basis taking values in $L^2([0,T]\times \mathbb{T}^N)\cap C([0,T], H^{-(1+\epsilon)}(\mathbb{T}^N))$ such that $\tilde{\mathbb{P}}$-almost surely
        \[
        \tilde{u}_\delta\to \tilde{u} \qquad  \text{in}\  L^2([0,T]\times \mathbb{T}^N)\cap C([0,T], H^{-(1+\epsilon)}(\mathbb{T}^N))
        \]
        as $\delta \to 1$. Note moreover that 
        \[
        \tilde{\mathbb{E}}\left[\int_0^T\norm{\tilde{u_s}^\delta}_{W^{1,2}(\mathbb{T}^N)}^2\right]=\mathbb{E}\left[\int_0^T\norm{u^\delta_s}_{W^{1,2}(\mathbb{T}^N)}^2\right]<\infty
        \]
        and thus in particular $\tilde{u}\in L^2(\Omega\times [0,T], W^{1,2}(\mathbb{T}^N))$ and $\tilde{u}^\delta\rightharpoonup \tilde{u}$ in $L^2([0,T], W^{1,2}(\mathbb{T}^N))$, $\tilde{\mathbb{P}}$-almost surely and also in $L^2(\Omega)$ due to Vitali's convergence theorem. In particular, by definition for any $\varphi\in H^{1}(\mathbb{T}^N)$ and $t\leq T$ we have
        \begin{equation}
       \int_0^t \langle (-\Delta+1)\tilde{u}^\delta_s, \varphi\rangle_{L^2(\mathbb{T}^N)} ds\to \int_0^t \langle (-\Delta+1)\tilde{u}_s, \varphi\rangle_{L^2(\mathbb{T}^N)} ds
       \label{conv_in_L2}
        \end{equation}
      in $L^2(\Omega)$. Let us now define the sequence of $L^2(\mathbb{T}^N)$ valued $(\mathcal{F}_t)_t$-martingales $(M^\delta)_{\delta\in [0,1)\cap \mathbb{Q}}$ via
      \begin{equation*}
          \begin{split}
         M^\delta_t:&=(-\Delta+1)^{-(1+\epsilon)/2}\left( u^\delta_t-u_0-\int_0^t (\Delta-1)u^\delta_sds\right)\\
         &=\int_0^t\mu(-\Delta+1)^{(\delta-1-\epsilon)/2}B(u^\delta_s)dW_s.
          \end{split}
      \end{equation*}
     Since $\mathcal{L}(u^\delta)=\mathcal{L}(\tilde{u}^\delta)$, it is easy to see that 
     \[
      \tilde{M}^\delta_t:=(-\Delta+1)^{-(1+\epsilon)/2}\left( \tilde{u}^\delta_t-u_0-\int_0^t (\Delta-1)\tilde{u}^\delta_sds\right)
      \]
      is a martingale with respect to the filtration $(\mathcal{G}_t)_t$, where $\mathcal{G}_t:=\sigma(\{\tilde{u}^\delta_s, s\leq t\})$, whose quadratic variation is given by 
      \[
      \langle \tilde{M}^\delta\rangle_t=\mu^2\int_0^t  \left((-\Delta+1)^{(\delta-1-\epsilon)/2}B(\tilde{u}^\delta_s)\right)\left((-\Delta+1)^{(\delta-1-\epsilon)/2}B(\tilde{u}^\delta_s)\right)^*ds
      \]
      Defining finally the process 
      \[
      \tilde{M}_t:= (-\Delta+1)^{-(1+\epsilon)/2}\left( \tilde{u}_t-u_0-\int_0^t (\Delta-1)\tilde{u}_sds\right),
      \]
      we need to show that it is a martingale with quadratic variation given by
      \[
      \langle \tilde{M}\rangle_t=\mu^2\int_0^t \ \left((-\Delta+1)^{-(1+\epsilon)/2}B(\tilde{u}_s)\right)\left((-\Delta+1)^{-(1+\epsilon)/2}B(\tilde{u}_s)\right)^*ds.
      \]
      Indeed, we know that for any bounded continuous functional $\phi$ on $L^2([0,T]\times \mathbb{T}^N)\cap C([0,T], H^{-(1+\epsilon)})$ and $v, z\in H^{1+\epsilon}(\mathbb{T}^N)$ we have
      \[
      \mathbb{E}\left[\langle \tilde{M}^\delta_t-\tilde{M}^\delta_s, v\rangle \phi(\tilde{u}^\delta|_{[0,s]})\right]=0
      \]
    and due to \eqref{conv_in_L2}, we can pass to the limit concluding that 
    \[
     \mathbb{E}\left[\langle \tilde{M}_t-\tilde{M}_s, v\rangle \phi(\tilde{u}|_{[0,s]})\right]=0
      \]
     i.e. $(\tilde{M}_t)_t$ is a martingale with respect to the filtration generated by the process $(\tilde{u}_t)_t$. Moreover, concerning its quadratic variation, note that 
      \begin{equation*}
          \begin{split}
                 &\mathbb{E}\left[ \left(\langle \tilde{M}^\delta_t,v\rangle\langle \tilde{M}^\delta_t,z\rangle-\langle \tilde{M}^\delta_s,v\rangle\langle \tilde{M}^\delta_s,z\rangle\right)\phi(\tilde{u}^\delta|_{[0,s]}) \right]\\
                 =&\mathbb{E}\Big[\left(\int_s^t  \sum_{i=1}^d\langle (-\Delta+1)^{(\delta-1-\epsilon)/2} B_i(\tilde{u}^\delta_r), v\rangle \langle (-\Delta+1)^{(\delta-1-\epsilon)/2} B_i(\tilde{u}^\delta_r), z\rangle dr\right)\\
                 &\phi(\tilde{u}^\delta|_{[0,s]})\Big]\mu^2.
          \end{split}
      \end{equation*}
Since the family of operators 
   \[
   u\to (-\Delta+1)^{(\delta-1-\epsilon)/2} B_i(u)
   \]
for $\delta\in [0,1)$ is uniformly continuous on $L^2(\mathbb{T}^N)$ and $\tilde{u}^\delta\to \tilde{u}$  in $L^2([0,T]\times \mathbb{T}^N)$ $\tilde{\mathbb{P}}$-almost surely and in $L^2(\Omega) $ due to Vitali's convergence theorem, we can again pass to the limit $\delta\nearrow 1$ obtaining 
\begin{equation*}
    \begin{split}
         &\mathbb{E}\left[ \left(\langle \tilde{M}_t,v\rangle\langle \tilde{M}_t,z\rangle-\langle \tilde{M}_s,v\rangle\langle \tilde{M}_s,z\rangle\right)\phi(\tilde{u}|_{[0,s]}) \right]\\
                 &=\mathbb{E}[\mu^2\left(\int_s^t  \sum_{i=1}^d\langle (-\Delta+1)^{-\epsilon/2} B_i(\tilde{u}_r), v\rangle \langle (-\Delta+1)^{-\epsilon/2} B_i(\tilde{u}_r), z\rangle dr\right)\phi(\tilde{u}|_{[0,s]})]
    \end{split}
\end{equation*}
      meaning that indeed 
      \[
      \langle \tilde{M}\rangle_t=\int_0^t \ \left((-\Delta+1)^{-(1+\epsilon)/2}B(\tilde{u}_s)\right)\left((-\Delta+1)^{-(1+\epsilon)/2}B(\tilde{u}_s)\right)^*ds.
      \]
      
            Similar to section 8.4 in \cite{daprato_zabczyk}, we can then use the martingale representation theorem 8.2 in \cite{daprato_zabczyk} to conclude the existence of a filtered probability space $(\Omega', \mathcal{F}', (\mathcal{F}'_t)_t, \mathbb{P}')$, a $d$-dimensional $(\mathcal{F}'_t)_t$-Brownian motion $W'$ and a predictable process $(u'_t)_t$ taking values in $L^2([0,T]\times \mathbb{T}^N)\cap C([0,T], H^{-(1+\epsilon)}(\mathbb{T}^N))$ such that 
      \begin{equation*}
          \begin{split}
               (-\Delta+1)^{-(1+\epsilon)/2}u'_t&=(-\Delta+1)^{-(1+\epsilon)/2}u_0\\
               &+(-\Delta+1)^{-(1+\epsilon)/2}\int_0^t (\Delta-1)u'_sds\\
               &+\mu(-\Delta+1)^{-(1+\epsilon)/2}\int_0^t(-\Delta+1)^{1/2} B(u'_s)dW'_s
          \end{split}
      \end{equation*}
     or in other words for any $\varphi\in H^{1+\epsilon}(\mathbb{T}^N)$, we have
     \begin{equation*}
         \begin{split}
              \langle u'_t, \varphi\rangle=\langle u_0, \varphi\rangle +\left\langle \int_0^t (\Delta-1)u'_sds, \varphi\right\rangle
               +\left\langle \int_0^t\mu(-\Delta+1)^{1/2} B(u'_s)dW'_s, \varphi\right\rangle
         \end{split}
     \end{equation*}
     meaning we have indeed a martingale solution.
      
       \begin{example}
       As an example one can consider the problem
       \begin{align}
           \begin{cases}
           du_t&= (\Delta-1) u_tdt+\mu \sum_{i=1}^d \mbox{div} (B_i(u_t))d\beta^i_t\\
           u(0)&=u_0
           \end{cases}
           \label{example}
       \end{align}
      for $B_1, \ldots,
      B_d\in C^1(\R, \R^N)$ with bounded derivative which can be seen as
        \begin{align*}
           \begin{cases}
           du_t&= (\Delta-1) u_tdt+\mu (-\Delta+1)^{1/2}\sum_{i=1}^d (-\Delta+1)^{-1/2}\mbox{div}(B_i(u_t))d\beta^i_t\\
           u(0)&=u_0.
           \end{cases}
       \end{align*}
       Having in mind the abstract results in the proof of Theorem \ref{main_spde_result}, one has to consider the Nemytskii operators
       \begin{align*}
           B^0: L^p(\mathbb{T}^N)&\rightarrow \gamma(U, L^p(\mathbb{T}^N))\\
           z &\rightarrow \left( u\rightarrow \sum_{i=1}^d (-\Delta+1)^{-1/2}\mbox{div} (B_i(z))\langle u, e_i\rangle\right) 
       \end{align*}
       and
       \begin{align*}
           B^1: W^{1,p}(\mathbb{T}^N)&\rightarrow \gamma(U, W^{1,p}(\mathbb{T}^N))\\
           z &\rightarrow \left( u\rightarrow \sum_{i=1}^d (-\Delta+1)^{-1/2}\mbox{div} (B_i(z))\langle u, e_i\rangle\right).
       \end{align*}
       After showing that $B^0$  is Lipschitz and $B^1$ admits a growth condition in the above setting, one can for each $\delta<1$ construct a unique mild solution $u^\delta\in Z_{q, W^{1, p}(\mathbb{T}^N)}$ to the regularized problem
       \begin{align*}
           \begin{cases}
           du^\delta_t&= (\Delta-1) u^\delta_t dt+\mu (-\Delta+1)^{\delta/2}\sum_{i=1}^d (-\Delta+1)^{-1/2}\mbox{div}B_i(u^\delta_t)dW^i_t\\
           u^\delta(0)&=u_0.
           \end{cases}
       \end{align*}
       Provided $\mu$ is sufficiently small, one can exploit the reasoning of this section to deduce the existence of a martingale solution to the problem \eqref{example}. 
       \end{example}
        \begin{remark}
     Existence of weak solutions in the  linear setting of this example i.e. $(B_i)_{i\leq d}$ linear is classical, see Example 7.22 in \cite{daprato_zabczyk} or section three in \cite{krylov_stochastic_1981}. In this light, the present subsection provides an alternative approximation scheme to Yosida approximations used in \cite{daprato_zabczyk} and Galerkin approximations employed in \cite{krylov_stochastic_1981}, capable of treating non-linear diffusion coefficients in this setting.
       \end{remark}

\section*{Appendix}
\begin{lemma}
Let $G\in C^m(\mathbb{T}^N\times \R)$ and $F\in C^m(\R)$ have bounded derivative and such that $F(0)=0$. Then for all $p\geq 1$ and $h\in W^{m,p}(\mathbb{T}^N)\cap W^{1, mp}(\mathbb{T}^N)$ one has
\[
||G(\cdot, h(\cdot))||_{W^{m,p}(\mathbb{T}^N)}\leq C(1+||h||^m_{W^{1,mp}(\mathbb{T}^N)}+||h||_{W^{m,p}(\mathbb{T}^N)})
\]
and 
\[
||F(h)||_{W^{m,p}(\mathbb{T}^N)}\leq C(||h||^m_{W^{1,mp}(\mathbb{T}^N)}+||h||_{W^{m,p}(\mathbb{T}^N)}).
\]
(refer to \cite{Hofmanova2013} and the references therein).
\label{Nemytskii_simple}
\end{lemma}
\subsection*{On stochastic integrals in different Banach spaces}
Note that the stochastic integral in a Banach space $X$ is defined as a certain limit in the Bochner space $L^2(\Omega; X)$, i.e. it depends in particular on the topology induced by the norm on $X$. To underline this fact, introduce for $\psi\in L^2(\Omega\times[0,T]; \gamma(U, X))$ the notation
\[
(X\int_0^T)\psi(t)dW_t\in L^2(\Omega; X)
\]
Heuristically speaking, one should expect the stochastic integral not to change if one looks at it in a "larger" Banach space. A bit more formally, for $X\hookrightarrow Y$, one would expect
\[
(X\int_0^T)\psi(t)dW_t=(Y\int_0^T)\psi(t)dW_t.
\]
The following Lemma formalizes this consideration. 

\begin{lemma}(Banach space consistency of stochastic integration)
Let $X,Y$ be 2-smooth Banach spaces such that $X\hookrightarrow Y$, where the embedding operator is the identity operator. Suppose that $\psi\in L^2(\Omega\times [0,T]; \gamma(U,X))$. Then one has $\psi\in L^2(\Omega\times [0,T]; \gamma(U,Y))$ and thus the stochastic integral
\[
 (Y\int_0^T)\psi(t)dW_t\in L^2(\Omega; Y)
\]
is well defined. Moreover this $Y$-stochastic integral also lies in  $L^2(\Omega; X)$ and one has
\[
(X\int_0^T)\psi(t)dW_t= (Y\int_0^T)\psi(t)dW_t \qquad \mbox{in } L^2(\Omega; X)
\]
\label{Banach_space_consistency}
\end{lemma}
\begin{proof}
Note that  by the continuous embedding $X\hookrightarrow Y$we have $\gamma(U,X)\hookrightarrow \gamma(U,Y)$ and therefore also $L^2(\Omega\times [0,T]; \gamma(U,X))\hookrightarrow L^2(\Omega\times [0,T]; \gamma(U,Y))$, i.e. in particular $\psi\in  L^2(\Omega\times [0,T]; \gamma(U,Y))$.\newline\indent
For the second part of the Lemma, note that stochastic integrals are defined as $L^2$ limits of stochastic integrals over approximating elementary processes. On the level of elementary processes, the norms of $X$ and $Y$ do not come into play and the canonical definitions of stochastic integrals with respect to elementary processes coincide therefore in $L^2(\Omega, X)$ and $L^2(\Omega, Y)$ as the spaces of reference in which the stochastic integrals live. \newline\indent
Let $(\psi^n)_n$ be a sequence of elementary processes approximating $\psi\in L^2(\Omega\times [0,T]; \gamma(U; X))$, i.e. one has
 \[
 \mathbb{E}\left[\int_0^T||\psi(t)-\psi^n(t)||_{\gamma(U, X)}^2ds\right]\rightarrow 0
 \]
 Due to the demonstrated embedding result one also has
 $(\psi^n)_n\subset L^2(\Omega\times [0,T]; \gamma(U; Y))$ and moreover, for the same reason
 \begin{align*}
      \mathbb{E}\left[\int_0^T||\psi(t)-\psi^n(t)||^2_{\gamma(U,Y)}dt\right]&\leq C \mathbb{E}\left[\int_0^T||\psi(t)-\psi^n(t)||^2_{\gamma(U,X)}dt\right]\\
      &\rightarrow 0.
 \end{align*}
 Hence, any sequence of elementary processes $(\psi^n)_n$ approximating $\psi$ in the space $L^2(\Omega\times [0,T]; \gamma(U; X))$ also approximates $\psi$ in $L^2(\Omega\times [0,T]; \gamma(U; Y))$. 
 Moreover, by Itô's inequality applied in the Banach space $Y$ one has 
 \begin{align*}
    \mathbb{E}\left[||\int_0^T\psi^n(t)dW_s-\int_0^T\psi^m(t)dW_s||_Y^2\right]&\leq \mathbb{E}\left[\int_0^T ||\psi^n(t)-\psi^m(t)||_{\gamma(U,Y)}^2ds\right]\\
    &\rightarrow 0 
 \end{align*}
 The sequence 
 \[
 \left(\int_0^t\psi^n(t)dW_s\right)_{n\geq 1}
 \]
 is therefore Cauchy in $L^2(\Omega, Y)$.
 By definition, the stochastic integral $I_Y:=(Y\int_0^T)\psi(t)dW_t$ is the $L^2(\Omega; Y)$ limit of the above sequence. Note however that
 \begin{align*}
 &\mathbb{E}\left[||(X\int_0^T)\psi(t)dW_s-\int_0^T\psi^n(t)dW_s||_Y^2\right]\\
 \leq& \mathbb{E}\left[||(X\int_0^T)\psi(t)dW_s-\int_0^T\psi^n(t)dW_s||_X^2\right]\rightarrow 0
 \end{align*}
 since $(\psi^n)_n$ is a sequence of elementary processes used to define the stochastic integral in the Banach space $X$. By the uniqueness of limits, one concludes therefore
 \[
 (X\int_0^T)\psi(t)dW_s=(Y\int_0^T)\psi(t)dW_s.
 \]
 \end{proof}
\section*{Acknowledgment}
The author wishes to thank Martina Hofmanová, Arnaud Debussche, Máté Gerencsér, Tommaso Rosati and his PhD advisor Lorenzo Zambotti for valuable discussions as well as the anonymous referee for her/his detailed report pointing out some mistakes in an earlier version. This project has received funding from the European Union's Horizon 2020 research and innovation programme under the Marie Skłodowska-Curie grant agreement No 754362.

	\begin{figure}[ht]
	    \centering
	  \includegraphics[height=10mm]{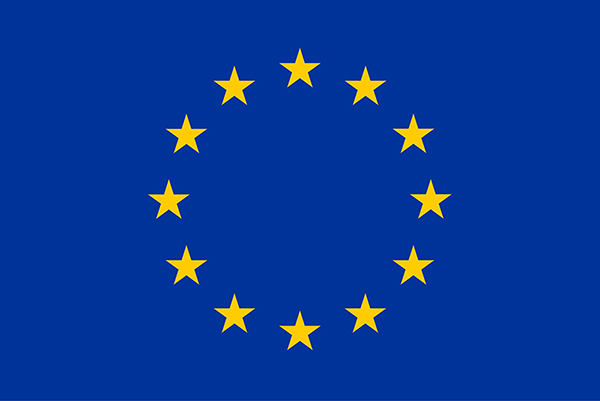}
	\end{figure}

     \bibliographystyle{unsrt}
\bibliography{main}
 \end{document}